\documentclass{amsart}

\usepackage{bbding}
\usepackage{amsmath}
\usepackage{amssymb,amsfonts,amsthm}
\usepackage[all]{xy}

\newtheorem{theorem}{Theorem}[section]
\newtheorem{lemma}[theorem]{Lemma}
\newtheorem{proposition}[theorem]{Proposition}
\newtheorem{corollary}[theorem]{Corollary}
\theoremstyle{definition}
\newtheorem{definition}[theorem]{Definition}
\newtheorem{example}[theorem]{Example}

\newtheorem{remark}[theorem]{Remark}
\numberwithin{equation}{section}

\newcommand{\blankbox}[2]

\begin{document}
\title{Extending structures for associative conformal algebras}
\author{Yanyong Hong}
\address{College of Science, Zhejiang Agriculture and Forestry University,
Hangzhou, 311300, P.R.China}
\email{hongyanyong2008@yahoo.com}

\thanks{Project supported by the Zhejiang Provincial Natural Science Foundation of China (No. LQ16A010011), the National Natural Science Foundation of China (No. 11501515 and No. 11626216), and the Scientific Research Foundation of Zhejiang Agriculture and Forestry University (No. 2013FR081)}
\begin{abstract}
In this paper, we give a study of the $\mathbb{C}[\partial]$-split extending structures problem for associative conformal algebras. Using the unified product as a tool, which includes interesting products such as bicrossed product, cocycle semi-direct product and so on, a cohomological type object is constructed to characterize the $\mathbb{C}[\partial]$-split extending structures for associative conformal algebras. Moreover, using this theory, the extending structures of an associative conformal algebra $A$ which is free as a $\mathbb{C}[\partial]$-module by
the $\mathbb{C}[\partial]$-module $Q=\mathbb{C}[\partial]x$ are described using flag datums of $A$. Furthermore, we give a classification of the extending structures of $A$ by $Q=\mathbb{C}[\partial]x$ in detail up to equivalence when $A$ is a free associative conformal algebra of rank 1.

\end{abstract}

\maketitle

\section{Introduction}
Conformal algebras  were introduced in \cite{K1}, \cite{K2} as a useful tool to study vertex algebras.
The structure of a (Lie) conformal algebra gives an
axiomatic description of the operator product expansion (or rather
its Fourier transform) of chiral fields in conformal field theory.
Structure theory and representation theory of Lie and associative conformal algebras are studied in a series of papers (see \cite{BKV}-\cite{DK1}, \cite{Ko1}-\cite{Z2}). In particular, associative conformal algebras naturally appear in the representation theory of Lie conformal algebras. Moreover, conformal algebras have close connections to infinite-dimensional algebras satisfying the locality property (\cite{K}).

In this paper, following the work on the extending structures problem of Lie conformal algebras (see \cite{HS}), we study the version of associative conformal algebras.\\
{\bf The $\mathbb{C}[\partial]$-split extending structures problem:}~~~\emph{Let $A$ be an associative conformal algebra and $Q$ be a $\mathbb{C}[\partial]$-module. Set $E=A\oplus Q$ where the direct sum is the sum of $\mathbb{C}[\partial]$-modules. Describe and classify all associative conformal algebra structures on $E$ up to isomorphism such that $A$ is a subalgebra of $E$.}\\
Similar problems for some classical algebra objects such as groups, associative algebras, Hopf algebras, Lie algebras, Leibniz algebras and left-symmetric algebras have been studied in \cite{AM1, AM3, AM2, AM4, AM6, H1} respectively.
Note that this problem is very difficult. When $A=0$, it is equivalent to classifying all associative conformal algebras of arbitrary ranks. As we know, it is very hard to classify those torsion-free associative conformal algebras whose rank is larger than 1. Thus, in this paper, we always assume that $A\neq 0$.
This problem includes many important problems in the structure theory of associative conformal algebra. For example,
when $Q$ is an associative conformal algebra, the problem that how to describe and classify all associative conformal algebra structures on $E$ such that $A$ and $Q$ are two subalgebras of $E$ up to isomorphism  is a special case of the $\mathbb{C}[\partial]$-split extending structures problem. Moreover, it also includes the following problem:\\
{\bf The $\mathbb{C}[\partial]$-split extension problem:}~~~\emph{Given two associative conformal algebras $A$ and $Q$. Describe and classify
all $\mathbb{C}[\partial]$-split exact sequences of associative conformal algebras as follows up to equivalence:
\begin{eqnarray}\label{en1}
\xymatrix@C=0.5cm{
  0 \ar[r] & A\ar[rr]^{i} && E\ar[rr]^{\pi} && Q\ar[r] & 0 }.
\end{eqnarray}}
Of course, the $\mathbb{C}[\partial]$-split central extension problem of an associative conformal algebra $Q$ by a trivial associative conformal algebra $A$ also belongs to this problem. It is known from \cite{BKV} that all such $\mathbb{C}[\partial]$-split central extensions up to equivalence
can be characterized by the second cohomology group $H^2(Q,A)$. Therefore, the study of the $\mathbb{C}[\partial]$-split extending structures problem is meaningful and can be used to investigate the structure of associative conformal algebras. In this paper, we hope to construct some cohomological type object to describe and classify the $\mathbb{C}[\partial]$-split extending structures of an associative conformal algebra and give a detailed study of structure theory of associative conformal algebras by using this cohomological type object. It will be found that this theory can provide an efficient way to classify some associative conformal algebras of low ranks.

This paper is organized as follows. In Section 2, we recall some preliminaries about associative conformal algebras, including the definitions of associative conformal algebra and its module. In Section 3, we introduce the unified product $A\natural Q$ associated with an extending datum $\Omega(A, Q)=(\leftharpoonup_\lambda, \rightharpoonup_\lambda, \lhd_\lambda, \rhd_\lambda,  f_\lambda, \circ_\lambda )$ where $A$ is an associative conformal algebra and $Q$ is a $\mathbb{C}[\partial]$-module.  It is shown that any associative conformal algebra $E$ satisfying the condition in the $\mathbb{C}[\partial]$-split extending structures problem is isomorphic to a unified product of $A$ and $Q$. Later on, a cohomological type object $\mathcal{AH}_A^2(Q,A)$
is constructed to give a theoretical answer to the $\mathbb{C}[\partial]$-split extending structures problem for associative conformal algebras by describing the isomorphisms between two different unified products of $A$ and $Q$ which
stabilizes $A$. It should be pointed out that the unified product of associative conformal algebra is more complex than that of Lie conformal algebra. It requires two more actions, since it does not have skew-symmetry. In Section 4, we introduce some special cases of unified products such as cocycle semi-direct product, crossed product, bicrossed product and so on. Moreover, the unified product of commutative conformal algebra is also presented. In Section 5, using the general theory, the extending structures of an associative conformal algebra $A$ which is free as a $\mathbb{C}[\partial]$-module by
the $\mathbb{C}[\partial]$-module $Q=\mathbb{C}[\partial]x$ are described using flag datums of $A$. In addition,  a classification of the extending structures of $A$ by $Q=\mathbb{C}[\partial]x$ in detail up to equivalence is given when $A$ is a free associative conformal algebra of rank 1.

Throughout this paper, denote by $\mathbb{C}$ the field of complex
numbers; $\mathbb{N}$ is the set of natural numbers, i.e.,
$\mathbb{N}=\{0, 1, 2,\cdots\}$; $\mathbb{Z}$ is the set of integer
numbers. All tensors over $\mathbb{C}$ are denoted by $\otimes$.
Moreover, if $A$ is a vector space, the space of polynomials of $\lambda$ with coefficients in $A$ is denoted by $A[\lambda]$.

\section{Preliminaries}
In this section, we recall some definitions and results about associative conformal algebras. These facts can be referred to \cite{K1}.
\begin{definition}
An \emph{associative conformal algebra} $A$ is a $\mathbb{C}[\partial]$-module with a $\lambda$-product $\cdot_\lambda \cdot$ which defines a $\mathbb{C}$-bilinear
map from $A\times A\rightarrow A[\lambda]$ satisfying
\begin{eqnarray*}
&&(\partial a)_\lambda b=-\lambda a_\lambda b,~~~a_\lambda \partial b=(\lambda+\partial)a_\lambda b, ~~\text{(conformal sesquilinearity)}\\
&&a_\lambda(b_\mu c)=(a_\lambda b)_{\lambda+\mu} c,~~\text{(associativity)}
\end{eqnarray*}
for $a$, $b$, $c\in A$.

An associative conformal algebra $A$ is called \emph{commutative} if for all $a$, $b\in A$,
\begin{eqnarray*}
a_\lambda b=b_{-\lambda-\partial}a.
\end{eqnarray*}
\end{definition}

An associative conformal algebra is called \emph{finite} if it is finitely generated as a $\mathbb{C}[\partial]$-module. The \emph{rank} of an associative conformal
algebra $A$ is its rank as of a $\mathbb{C}[\partial]$-module.
The notions of a homomorphism, ideal and subalgebra of an associative conformal
algebra are defined as usual.

\begin{example}

Let $(A,\circ)$ be an associative algebra. The current associative conformal
algebra associated to $A$ is defined by
$$\text{Cur} A=\mathbb{C}[\partial]\otimes A, ~~a_\lambda b=a\circ b,
~~a,b\in A.$$
\end{example}

Let $A$ be an associative conformal algebra, $Q$ a $\mathbb{C}[\partial]$-module and $E=A\oplus Q$ where the sum is the direct sum of $\mathbb{C}[\partial]$-modules. For a $\mathbb{C}[\partial]$-module homomorphism $\varphi: E\rightarrow E$, we consider the following
diagram
$$\xymatrix{ {A}\ar[d]^{Id}\ar[r]^{i}& {E}\ar[d]^{\varphi}\ar[r]^{\pi} & {Q}\ar[d]^{Id} \\
{A}\ar[r]^{i}&{E}\ar[r]^{\pi} &{Q} },$$
where $\pi: E\rightarrow Q$ is the natural projection of $E=A\oplus Q$ onto $Q$ and
$i: A\rightarrow E$ is the inclusion map. We say that
$\varphi: E\rightarrow E$ \emph{stabilizes} $A$ (resp. \emph{co-stabilizes} $Q$) if the left square
(resp. the right square) of the  above diagram is commutative.

Let $\cdot_\lambda$ and $\circ_\lambda$ be two associative conformal algebra structures on $E$ both containing $A$ as an associative conformal subalgebra.
If there exists an associative conformal algebra isomorphism $\varphi: (E, \cdot_\lambda )\rightarrow (E,\circ_\lambda)$ which stabilizes
$A$,  then $\cdot_\lambda$ and $\circ_\lambda$ are called \emph{equivalent}, which is denoted by
$(E, \cdot_\lambda )\equiv (E,\circ_\lambda)$.

If there exists an associative conformal algebra isomorphism $\varphi: (E, \cdot_\lambda)\rightarrow (E,\circ_\lambda)$ which stabilizes
$A$ and co-stabilizes $Q$,   $\cdot_\lambda$ and $\circ_\lambda$ are called \emph{cohomologous}, which is denoted by
$(E,\cdot_\lambda)\approx (E,\circ_\lambda)$.

Obviously, $\equiv $ and $\approx $ are equivalence relations on the set of all associative conformal algebra structures on $E$ containing $A$ as an associative conformal subalgebra. Denote  the set
of all equivalence classes via $\equiv $ (resp. $\approx $) by $\text{CExtd}(E,A)$ (resp. $\text{CExtd}^{'}(E,A)$). It is easy to see that $\text{CExtd}(E,A)$ is the classifying object of the
$\mathbb{C}[\partial]$-split extending structures problem and $\text{CExtd}^{'}(E,A)$ provides a classification of the $\mathbb{C}[\partial]$-split
 extending structures  from the point of the view of the extension problem. Moreover, there exists a canonical projection $\text{CExtd}^{'}(E,A)\twoheadrightarrow \text{CExtd}(E,A) $. Therefore, for studying the
$\mathbb{C}[\partial]$-split extending structures problem, we only need to study  $\text{CExtd}(E,A)$ and $\text{CExtd}^{'}(E,A)$.

For characterizing $\text{CExtd}(E,A)$ and $\text{CExtd}^{'}(E,A)$, we need to introduce the definitions of modules over associative conformal algebras.
\begin{definition}
A \emph{left module} $M$ over an associative conformal algebra $A$ is a $\mathbb{C}[\partial]$-module endowed with a $\mathbb{C}$-bilinear map
$A\times M\longrightarrow M[\lambda]$, $(a, v)\mapsto a\rightharpoonup_\lambda v$, satisfying the following axioms $(a, b\in A, v\in M)$:\\
(LM1)$\qquad\qquad (\partial a)\rightharpoonup_\lambda v=-\lambda a\rightharpoonup_\lambda v,~~~a\rightharpoonup_\lambda(\partial v)=(\partial+\lambda)a\rightharpoonup_\lambda v,$\\
(LM2)$\qquad\qquad (a_\lambda b)\rightharpoonup_{\lambda+\mu}v=a\rightharpoonup_\lambda(b\rightharpoonup_\mu v).$\\
We also denote it by $(M,\rightharpoonup_\lambda)$.

A \emph{right module} $M$ over an associative conformal algebra $A$ is a $\mathbb{C}[\partial]$-module endowed with a $\mathbb{C}$-bilinear map
$M\times A\longrightarrow M[\lambda]$, $(v, a)\mapsto v\lhd_\lambda a$, satisfying the following axioms $(a, b\in A, v\in M)$:\\
(RM1)$\qquad\qquad (\partial v)\lhd_\lambda a=-\lambda v\lhd_\lambda a,~~~v\lhd_\lambda(\partial a)=(\partial+\lambda)v\lhd_\lambda a,$\\
(RM2)$\qquad\qquad (v\lhd_\lambda a)_{\lambda+\mu}b=v\lhd_\lambda(a_\mu b).$\\
Usually, we denote it by $(M,\lhd_\lambda)$.

An $A$-\emph{bimodule} is $(M,\rightharpoonup_\lambda,\triangleleft_\lambda)$ such that $(M,\rightharpoonup_\lambda)$ is a left $A$-module, $(M,\lhd_\lambda)$ is a right $A$-module, and
they satisfy the following condition
\begin{eqnarray}
(a\rightharpoonup_\lambda v)\lhd_{\lambda+\mu}b=a\rightharpoonup_\lambda(v\lhd_\mu b),
\end{eqnarray}
where $a$, $b\in A$ and $v\in M$.
\end{definition}

\begin{definition}
Let $U$ and $V$ be two $\mathbb{C}[\partial]$-modules. A \emph{left conformal linear map} from $U$ to $V$ is a $\mathbb{C}$-linear map $a: U\rightarrow V[\lambda]$, denoted by $a_\lambda: U\rightarrow V[\lambda]$, such that $a_\lambda(\partial u)=-\lambda a_\lambda u$. Similarly, a \emph{right conformal linear map} from $U$ to $V$ is a $\mathbb{C}$-linear map $a: U\rightarrow V[\lambda]$, denoted by $a_\lambda: U\rightarrow V[\lambda]$, such that $a_\lambda(\partial u)=(\partial+\lambda)a_\lambda u$. A right conformal linear map is usually called as conformal linear map in short.

Moreover, let $W$ also be a $\mathbb{C}[\partial]$-module. A \emph{conformal bilinear map} from $U\times V$ to $W$ is a $\mathbb{C}$-bilinear map $f: U\times V \rightarrow W[\lambda]$, denoted by $f_\lambda: U\times V\rightarrow W[\lambda]$, such that $f_\lambda(\partial u,v)=-\lambda f_\lambda(u,v)$ and $f_\lambda( u,\partial v)= (\partial+\lambda)f_\lambda(u,v)$.
We say a conformal bilinear map from $U\times U$ to $W$ is \emph{symmetric}, if for all $u_1$, $u_2\in U$,
\begin{eqnarray}
f_\lambda(u_1,u_2)=f_{-\lambda-\partial}(u_2,u_1).
\end{eqnarray}

In addition, denote $\text{Cend}(V)$ by the set of all conformal linear maps from $V$ to $V[\lambda]$.
It has a canonical structure of a $\mathbb{C}[\partial]$-module:
$$(\partial a)_\lambda =-\lambda a_\lambda.$$
When $V$ is a finite module, $\text{Cend}(V)$ has a canonical structure of an associative conformal algebra defined by
\begin{eqnarray}
(a_\lambda b)_\mu v=a_\lambda (b_{\mu-\lambda} v),~~~~\text{ $a$, $b\in \text{Cend}(V)$, $v\in V$.}\end{eqnarray}
Denote $\text{Cend}(\mathbb{C}[\partial]^N)$ by $\text{Cend}_N$.

\end{definition}

\begin{proposition}\label{ppp2}
Let $A=\mathbb{C}[\partial]e$ be an associative conformal algebra which is free
and of rank 1 as a $\mathbb{C}[\partial]$-module.
Then $A$ is either trivial or isomorphic to $\text{Cur}B$ where $B=\mathbb{C}e$ with
$e\circ e=e$.
\end{proposition}
\begin{proof}
It is easy to see that such an algebra $A$ acts on itself either trivially or faithfully.
Obviously, in the first case, $A$ is trivial. In the second case, $A\subset \text{Cend}_1$ and
by Theorem 2.1 in \cite{BKL}, $A$ isomorphic to $\text{Cur}B$.
\end{proof}

\section{Unified products for associative conformal algebras}
In this section, we will introduce the unified products for associative conformal algebras and use this product to provide a theoretical answer to the
$\mathbb{C}[\partial]$-split extending structures problem.

\begin{definition}
Let $(A,\cdot_\lambda \cdot)$ be an associative conformal algebra  and $Q$ a $\mathbb{C}[\partial]$-module. An \emph{extending datum} of $A$ by $Q$ is
a system $\Omega(A, Q)=(\leftharpoonup_\lambda, \rightharpoonup_\lambda, \lhd_\lambda, \rhd_\lambda, $$ f_\lambda, $ $\circ_\lambda )$ consisting six conformal bilinear maps given by
\begin{eqnarray*}
\leftharpoonup_\lambda: A\times Q\rightarrow A[\lambda],~~~\rightharpoonup_\lambda:
A\times Q\rightarrow Q[\lambda],\\
\lhd_\lambda: Q\times A\rightarrow Q[\lambda],~~~\rhd_\lambda: Q\times A\rightarrow A[\lambda],\\
f_\lambda: Q\times Q \rightarrow A[\lambda],~~~\circ_\lambda: Q\times Q \rightarrow Q[\lambda].
\end{eqnarray*}
Let $\Omega(A, Q)=(\leftharpoonup_\lambda, \rightharpoonup_\lambda,\lhd_\lambda, \rhd_\lambda,  f_\lambda, \circ_\lambda )$ be an extending datum. Denote by $A\natural_{\Omega(A, Q)}$ $Q$ $=A\natural Q$ the $\mathbb{C}[\partial]$-module
$A\times Q$ with the natural $\mathbb{C}[\partial]$-module action: $\partial (a, x)=(\partial a, \partial x)$ and the bilinear map $
\cdot_\lambda \cdot: (A\times Q)\times (A\times Q)\rightarrow (A\times Q)[\lambda]$ defined by
\begin{gather}
(a,x)_\lambda (b,y)\nonumber\\
\label{e3}=(a_\lambda b+a\leftharpoonup_\lambda y+x\rhd_\lambda b+f_\lambda(x,y),a\rightharpoonup_\lambda y+x\lhd_\lambda b+x\circ_\lambda y),
\end{gather}
for all $a$, $b\in A$, $x$, $y\in Q$. Since $\lhd_\lambda$, $\rhd_\lambda$, $\leftharpoonup_\lambda$, $\rightharpoonup_\lambda$, $f_\lambda$ and $\circ_\lambda$ are conformal bilinear maps, the $\lambda$-product defined by (\ref{e3}) satisfies conformal sesquilinearity. Then
$A\natural Q$ is called the \emph{unified product} of $A$ and $Q$ associated with $\Omega(A, Q)$ if it is an associative conformal algebra with the $\lambda$-product given by (\ref{e3}). In this case, the extending datum $\Omega(A, Q)$ is called an \emph{associative conformal extending structure} of $A$ by $Q$.
We denote by $\mathcal{TC}(A,Q)$ the set of all associative conformal extending structures of $A$ by $Q$.
\end{definition}

By (\ref{e3}), the following equalities hold in $A\natural Q$ for all $a$, $b\in A$, $x$, $y\in Q$:
\begin{eqnarray}
&&\label{er1}(a,0)_{\lambda}(b,0)=(a_{\lambda} b,0),~~~(a,0)_{\lambda}(0,y)=(a\leftharpoonup_\lambda y, a\rightharpoonup_\lambda y),\\
&&\label{er2}(0,x)_{\lambda}(b,0)=(x\rhd_{\lambda} b, x\lhd_{\lambda}b),~~~(0,x)_{\lambda} (0,y)=(f_{\lambda}(x,y),x\circ_{\lambda} y).
\end{eqnarray}

\begin{theorem}\label{t1}
Let $A$ be an associative conformal algebra, $Q$ be a $\mathbb{C}[\partial]$-module and $\Omega(A, Q)$ an extending datum of $A$ by $Q$.
Then $A\natural Q$ is an associative conformal algebra if and only if the following conditions are satisfied for all
$a$, $b\in A$ and $x$, $y$, $z\in Q$:
\begin{eqnarray*}
(ACE1)&&~~(a_\lambda b)\leftharpoonup_{\lambda+\mu}x=a_\lambda(b\leftharpoonup_\mu x)+a\leftharpoonup_\lambda(b\rightharpoonup_\mu x),\\
(ACE2)&&~~a\rightharpoonup_\lambda(b\rightharpoonup_\mu x)=(a_\lambda b)\rightharpoonup_{\lambda+\mu}x,\\
(ACE3)&&~~a_\lambda(x\rhd_\mu b)+a\leftharpoonup_\lambda(x\lhd_\mu b)
=(a\leftharpoonup_\lambda x)_{\lambda+\mu}b+(a\rightharpoonup_\lambda x)\rhd_{\lambda+\mu}b,\\
(ACE4)&&~~a\rightharpoonup_\lambda(x\lhd_\mu b)=(a\rightharpoonup_\lambda x)\lhd_{\lambda+\mu}b,\\
(ACE5)&&~~x\rhd_\lambda(a_\mu b)=(x\rhd_\lambda a)_{\lambda+\mu}b+(x\lhd_\lambda a)\rhd_{\lambda+\mu}b,\\
(ACE6)&&~~x\lhd_\lambda(a_\mu b)=(x\lhd_\lambda a)\lhd_{\lambda+\mu}b,\\
(ACE7)&&~~a\leftharpoonup_\lambda(x\circ_\mu y)=(a\leftharpoonup_\lambda x)\leftharpoonup_{\lambda+\mu}y+f_{\lambda+\mu}(a\rightharpoonup_\lambda x,y)-a_\lambda f_\mu(x,y),\\
(ACE8)&&~~a\rightharpoonup_\lambda(x\circ_\mu y)=(a\rightharpoonup_\lambda x)\circ_{\lambda+\mu}y+(a\leftharpoonup_\lambda x)\rightharpoonup_{\lambda+\mu} y,\\
(ACE9)&&~~x\rhd_\lambda (a\leftharpoonup_\mu y)
+f_\lambda(x,a\rightharpoonup_\mu y)=(x\rhd_\lambda a)\leftharpoonup_{\lambda+\mu}y
+f_{\lambda+\mu}(x\lhd_\lambda a,y),\\
(ACE10)&&~~x\lhd_\lambda(a\leftharpoonup_\mu y)+x\circ_\lambda(a\rightharpoonup_\mu y)
=(x\rhd_\lambda a)\rightharpoonup_{\lambda+\mu}y+(x\lhd_\lambda a)\circ_{\lambda+\mu}y,\\
(ACE11)&&~~x\rhd_\lambda(y\rhd_\mu a)
+f_\lambda(x,y\lhd_\mu a)=f_\lambda(x,y)_{\lambda+\mu}a
+(x\circ_\lambda y)\rhd_{\lambda+\mu}a,\\
(ACE12)&&~~x\lhd_\lambda(y\rhd_\mu a)+x\circ_\lambda(y\lhd_\mu a)
=(x\circ_\lambda y)\lhd_{\lambda+\mu}a,\\
(ACE13)&&~~x\rhd_\lambda f_\mu(y,z)
+f_\lambda(x,y\circ_\mu z)=f_\lambda(x,y)\leftharpoonup_{\lambda+\mu}z+f_{\lambda+\mu}(x\circ_\lambda y,z),\\
(ACE14)&&~~x\lhd_\lambda f_\mu(y,z)+x\circ_\lambda(y\circ_\mu z)
=f_\lambda(x,y)\rightharpoonup_{\lambda+\mu}z+(x\circ_\lambda y)\circ_{\lambda+\mu}z.
\end{eqnarray*}

\end{theorem}
\begin{proof}
Obviously, we only need to prove that associativity for (\ref{e3}) holds if and only if $(ACE1)$-$(ACE14)$ hold.

Set
$$J((a,x),(b,y),(c,z))=(a,x)_\lambda((b,y)_\mu (c,z))-((a,x)_\lambda (b,y))_{\lambda+\mu}(c,z)$$
for all $a$, $b$, $c\in A$, $x$, $y$, $z\in Q.$

It is easy to see that associativity for (\ref{e3}) is satisfied if and only if
$J((a,0),(b,0),$ $(c,0))=0$,
$J((a,0),(b,0),(0,x))=0$, $J((a,0),(0,x),(b,0))=0$, $J((0,x),(a,0),$ $(b,0))=0$,
$J((a,0),(0,x),(0,y))=0$, $J((0,x),(a,0),(0,y))=0$, $J((0,x),(0,y),$ $(a,0))$$=0$,
and $J((0,x),(0,y),(0,z))=0$
 for all $a$, $b$, $c\in A$ and all $x$, $y$, $z\in Q$.
Obviously, $J((a,0),(b,0),(c,0))=0$ if and only if  $A$ is an associative conformal algebra.

Since
\begin{eqnarray*}
&&J((a,0),(b,0),(0,x))=\\
&&=(a,0)_\lambda(b\leftharpoonup_\mu x,b\rightharpoonup_\mu x)-(a_\lambda b,0)_{\lambda+\mu}(0,x)\\
&&=(a_\lambda(b\leftharpoonup_\mu x)+a\leftharpoonup_\lambda(b\rightharpoonup_\mu x),a\rightharpoonup_\lambda(b\rightharpoonup_\mu x))\\
&&-((a_\lambda b)\leftharpoonup_{\lambda+\mu}x,(a_\lambda b)\rightharpoonup_{\lambda+\mu}x)\\
&&=0,
\end{eqnarray*}
$J((a,0),(b,0),(0,x))=0$ if and only if $(ACE1)$ and $(ACE2)$ are satisfied.

Similarly, we can get the following: $J((a,0),(0,x),(b,0))=0$ if and only if $(ACE3)$ and $(ACE4)$ are satisfied;
$J((0,x),(a,0),(b,0))=0$  if and only if $(ACE5)$ and $(ACE6)$ are satisfied;
$J((a,0),(0,x),(0,y))=0$  if and only if $(ACE7)$ and $(ACE8)$ are satisfied;
$J((0,x),(a,0),(0,y))=0$  if and only if $(ACE9)$ and $(ACE10)$ are satisfied;
$J((0,x),(0,y),(a,0))=0$  if and only if $(ACE11)$ and $(ACE12)$ are satisfied;
$J((0,x),(0,y),(0,z))=0$  if and only if $(ACE13)$ and $(ACE14)$ are satisfied.
\end{proof}
\begin{remark}\label{rr1}
$(ACE2)$, $(ACE4)$ and $(ACE6)$ just mean that $(Q,\rightharpoonup_\lambda,\lhd_\lambda)$ is an $A$-bimodule.
\end{remark}

It is easy to see that all unified products of $A$ and $Q$ satisfy the conditions in the $\mathbb{C}[\partial]$-split extending structures problem. Next, we will show that any $E$ in this problem is isomorphic to a unified product of $A$ and $Q$.

\begin{theorem}\label{t2}
Let $A$ be an associative conformal algebra, $Q$ be a $\mathbb{C}[\partial]$-module and $E=A\oplus Q$ be the direct sum of $\mathbb{C}[\partial]$-modules. Suppose that $E$ has an associative conformal algebra structure $\cdot_\lambda \cdot$  such that $A$ is a subalgebra. Then
there exists an associative conformal extending structure $\Omega(A,Q)=(\leftharpoonup_\lambda, \rightharpoonup_\lambda,\lhd_\lambda,\rhd_\lambda, f_\lambda, \circ_\lambda )$ of
$A$ by $Q$ and an isomorphism of associative conformal algebras $E\approx A\natural Q$.
\end{theorem}
\begin{proof}
According to that $E=A\oplus Q$, there is a natural $\mathbb{C}[\partial]$-module homomorphism $p: E\rightarrow A$ such that
$p(a)=a$ for all $a\in A$. Then we define the following extending datum $\Omega(A,Q)=(\leftharpoonup_\lambda, \rightharpoonup_\lambda, \lhd_\lambda,\rhd_\lambda, f_\lambda, \circ_\lambda )$ of $A$ by $Q$:
\begin{eqnarray}
&&\leftharpoonup_\lambda:A \times Q\rightarrow A[\lambda],~~~~~~~a\leftharpoonup_\lambda x:=p(a_\lambda x),\\
&&\rightharpoonup_\lambda: A \times Q\rightarrow Q[\lambda],~~~~~~~a\rightharpoonup_\lambda x:=a_\lambda x-p(a_\lambda x),\\
&&\lhd_\lambda: Q\times A\rightarrow A[\lambda],~~~~~~~~x\lhd_\lambda a:=p(x_\lambda a),\\
&&\rhd_\lambda: Q\times A\rightarrow Q[\lambda],~~~~~~~~x\rhd_\lambda a:=x_\lambda a-p(x_\lambda a),\\
&&f_\lambda: Q\times Q\rightarrow A[\lambda],~~~~~~~~f_\lambda(x,y):=p(x_\lambda y),\\
&&\circ_\lambda: Q\times Q\rightarrow Q[\lambda],~~~~x\circ_\lambda y=x_\lambda y-p(x_\lambda y).
\end{eqnarray}
With the similar proof as that in Theorem 2.4 in \cite{AM2}, it is easy to show that $\Omega(A,Q)=(\leftharpoonup_\lambda, \rightharpoonup_\lambda,\lhd_\lambda,\rhd_\lambda, f_\lambda, \circ_\lambda )$ is an associative conformal extending structure and
$E\approx A\natural Q$ as associative conformal algebras.
\end{proof}

By Theorem \ref{t2}, for characterizing the $\mathbb{C}[\partial]$-split extending structures of $A$ by $Q$ up to equivalence, we only need to study the unified products of $A$ and $Q$ up to isomorphism which stabilizes $A$.

\begin{definition}\label{DD1}
Let $A$ be an associative conformal algebra and $Q$ a $\mathbb{C}[\partial]$-module. If there exists a pair of $\mathbb{C}[\partial]$-module homomorphisms $(u,v)$ where $u: Q\rightarrow A$, $v\in \text{Aut}_{\mathbb{C}[\partial]}(Q)$ such that
the associative conformal extending structure $\Omega(A,Q)=(\leftharpoonup_\lambda,\rightharpoonup_\lambda, \lhd_\lambda,\rhd_\lambda,
f_\lambda,\circ_\lambda )$ can be obtained from
another corresponding extending structure  $\Omega^{'}(A,Q)=(\leftharpoonup_\lambda^{'},\rightharpoonup_\lambda^{'}, \lhd_\lambda^{'},\rhd_\lambda^{'},
f_\lambda^{'}, \circ_\lambda^{'} )$ using
$(u,v)$ as follows:
\begin{gather}
\label{g1}a\rightharpoonup_\lambda x=v^{-1}(a\rightharpoonup_\lambda^{'}v(x)),\\
a\leftharpoonup_\lambda x=a_\lambda u(x)+a\leftharpoonup_\lambda^{'}v(x)-u(a\rightharpoonup_\lambda x),\\
x\lhd_\lambda a=v^{-1}(v(x)\lhd_\lambda^{'}a),\\
x\rhd_\lambda a=u(x)_\lambda a+v(x)\rhd_\lambda^{'}a-u(x\lhd_\lambda a),\\
x\circ_\lambda y=v^{-1}(u(x)\rightharpoonup_\lambda^{'}v(y))+v^{-1}(v(x)\lhd_\lambda^{'}u(y))+v^{-1}(v(x)\circ_\lambda^{'}v(y)),\\
\label{g6}f_\lambda(x,y)=u(x)_\lambda u(y)+u(x)\leftharpoonup_\lambda^{'}v(y)
+v(x)\rhd_\lambda^{'}u(y)+f_\lambda^{'}(v(x),v(y))-u(x\circ_\lambda y),
\end{gather}
for all $a\in A$, $x$, $y\in Q$, then $\Omega(A,Q)$ and  $\Omega^{'}(A,Q)$ are called \emph{equivalent} and  denote it
by $\Omega(A,Q)\equiv \Omega^{'}(A,Q)$.

In particular, $v=Id$, $\Omega(A,Q)$ and  $\Omega^{'}(A,Q)$ are called \emph{cohomologous} and denote it
by $\Omega(A,Q)\approx \Omega^{'}(A,Q)$.

\end{definition}

\begin{lemma}\label{l1}
Suppose that $\Omega(A,Q)=(\leftharpoonup_\lambda,\rightharpoonup_\lambda, \lhd_\lambda,\rhd_\lambda,
f_\lambda,\circ_\lambda )$ and $\Omega^{'}(A,Q)=($ $\leftharpoonup_\lambda^{'}$ $\rightharpoonup_\lambda^{'}$, $\lhd_\lambda^{'},\rhd_\lambda^{'},$ $f_\lambda^{'},$ $ \circ_\lambda^{'})$
are two associative conformal extending structures of $A$ by $Q$.
Let $A\natural Q$ and $A\natural^{'} Q$ be the corresponding unified products.
Then $A\natural Q \equiv  A\natural^{'} Q$ if and only if $\Omega(A,Q)\equiv \Omega^{'}(A,Q)$,
and $A\natural Q \approx  A\natural^{'} Q$ if and only if $\Omega(A,Q)\approx \Omega^{'}(A,Q)$.
\end{lemma}
\begin{proof}
Let $\varphi: A\natural Q\rightarrow A\natural^{'} Q$ be an isomorphism of associative conformal algebras which stabilizes $A$.
According to that $\varphi$ stabilizes $A$,  $\varphi(a,0)=(a,0)$. Then we can assume
$\varphi(a,x)=(a+u(x),v(x))$ where $u:Q\rightarrow A$, $v:Q\rightarrow Q$ are two linear maps.
Obviously, $\varphi$ is a $\mathbb{C}[\partial]$-module homomorphism if and only if
$u$, $v$ are two $\mathbb{C}[\partial]$-module homomorphisms.
Then, similar to that in Lemma 2.5 in \cite{AM2}, it is easy to check that
$\varphi$ is an
algebra isomorphism which stabilizes $A$ if and only if $v$ is a $\mathbb{C}[\partial]$-module isomorphism, and (\ref{g1})-(\ref{g6}) hold.
Moreover, it is obvious that an associative conformal algebra isomorphism $\varphi$ which stabilizes $A$ and costabilizes
$Q$ if and only if $v=Id_Q$.

Then, this lemma can be obtained by Definition \ref{DD1} and the above discussion.
\end{proof}

By the discussion above, the answer for the $\mathbb{C}[\partial]$-split extending structures problem of associative conformal algebra is given as follows.
\begin{theorem}\label{t4}
Let $A$ be an associative conformal algebra and $Q$ a $\mathbb{C}[\partial]$-module. Set $E=A\oplus Q$ where the direct sum is the sum of $\mathbb{C}[\partial]$-modules. Then we get\\
(1) Denote $\mathcal{AH}_{A}^2(Q,A):=\mathcal{TC}(A,Q)/\equiv$. Then the map
\begin{eqnarray}
\mathcal{AH}_{A}^2(Q,A)\rightarrow \text{CExtd}(E,A),~~~~\overline{\Omega(A,Q)}\rightarrow (A\natural Q,\cdot_\lambda \cdot)
\end{eqnarray}
is bijective, where $\overline{\Omega(A,Q)}$ is the equivalence class of $\Omega(A,Q)$ under $\equiv$.\\
(2) Denote $\mathcal{AH}^2(Q,A):=\mathcal{TC}(A,Q)/\approx$. Then the map
\begin{eqnarray}
\mathcal{AH}^2(Q,A)\rightarrow \text{CExtd}^{'}(E,A),~~~~\overline{\overline{\Omega(A,Q)}}\rightarrow (A\natural Q,\cdot_\lambda \cdot)
\end{eqnarray}
is bijective, where $\overline{\overline{\Omega(A,Q)}}$ is the equivalence class of $\Omega(A,Q)$ under $\approx$.\\

\end{theorem}
\begin{proof}
It can be directly obtained from Theorem \ref{t1}, Theorem \ref{t2} and Lemma \ref{l1}.
\end{proof}
\section{Special cases of unified products}
In this section, we introduce several special cases of unified products such as cocycle semidirect product, crossed product, bicrossed product and so on. Note that by Theorem \ref{t4}, all unified products of an associative conformal algebra $A$ and a
$\mathbb{C}[\partial]$-module $Q$ up to equivalence can be classified by the cohomological type object $\mathcal{AH}_{A}^2(Q,A)$. All products introduced in this section up to equivalence can also be described by the corresponding cohomological type objects.
\subsection{Cocycle semidirect products and semidirect sum}
Let $\Omega(A,Q)$ $=(\leftharpoonup_\lambda, \rightharpoonup_\lambda, \lhd_\lambda,\rhd_\lambda,
f_\lambda,\circ_\lambda )$ be an extending datum of an associative conformal algebra $A$ by a $\mathbb{C}[\partial]$-module $Q$ where
$\leftharpoonup_\lambda$ and $\rhd_\lambda$ trivial. Then
$\Omega(A,Q)=(\rightharpoonup_\lambda, \lhd_\lambda,
f_\lambda,\circ_\lambda )$ is an associative conformal extending structure of $A$ by $Q$ if and only if
$(Q,\rightharpoonup_\lambda,\lhd_\lambda)$ is an $A$-bimodule and  $\rightharpoonup_\lambda$, $\lhd_\lambda$,
$f_\lambda$, $\circ_\lambda$ satisfy (ACE14) and
\begin{eqnarray}
&&f_{\lambda+\mu}(a\rightharpoonup_\lambda x,y)=a_\lambda f_\mu(x,y),\\
&&a\rightharpoonup_\lambda(x\circ_\mu y)=(a\rightharpoonup_\lambda x)\circ_{\lambda+\mu}y,\\
&&f_\lambda(x,a\rightharpoonup_\mu y)=f_{\lambda+\mu}(x\lhd_\lambda a,y),\\
&&x\circ_\lambda(a\rightharpoonup_\mu y)
=(x\lhd_\lambda a)\circ_{\lambda+\mu}y,\\
&&f_\lambda(x,y\lhd_\mu a)=f_\lambda(x,y)_{\lambda+\mu}a,\\
&&x\circ_\lambda(y\lhd_\mu a)
=(x\circ_\lambda y)\lhd_{\lambda+\mu}a,\\
&&f_\lambda(x,y\circ_\mu z)=f_{\lambda+\mu}(x\circ_\lambda y,z),
\end{eqnarray}
where $a\in A$ and $x$, $y$, $z\in Q$. Denote this unified product by $A\natural^f Q$, which is called the \emph{cocycle
semidirect product of associative conformal algebras}. Note that, in this case, $A\natural^f Q$ as an
$A$-bimodule is the direct sum of $A$-bimodules $(A,\cdot_\lambda,\cdot_\lambda)$ and $(Q,\rightharpoonup_\lambda,\lhd_\lambda)$.

\begin{corollary}
Let $A$  be an associative conformal algebra and $Q$ be a $\mathbb{C}[\partial]$-module.
Set $E= A\oplus Q$, where the direct sum is the sum of $\mathbb{C}[\partial]$-modules.
Assume that $E$ has an associative conformal algebra structure such that $A$ is a subalgebra of $E$ and
the inclusion $A\hookrightarrow E$ has a retraction that is an $A$-bimodule map. Then
$E$ is isomorphic to a cocycle semidirect product  $A\natural^f Q$.
\end{corollary}
\begin{proof}
It can be directly obtained from Theorem \ref{t2}.
\end{proof}

Finally, let us consider the extending datum $\Omega(A,Q)=(\leftharpoonup_\lambda,\rightharpoonup_\lambda, \lhd_\lambda,\rhd_\lambda,
f_\lambda,\circ_\lambda )$ of an associative conformal algebra $A$ by a $\mathbb{C}[\partial]$-module $Q$  with
$\leftharpoonup_\lambda$, $\rhd_\lambda$ and $f_\lambda$ trivial. Then $\Omega(A,Q)=(\rightharpoonup_\lambda, \lhd_\lambda,
\circ_\lambda )$ is an associative conformal extending structure of $A$ by $Q$ if and only if
$(Q,\rightharpoonup_\lambda,\lhd_\lambda)$ is an $A$-bimodule, $(Q,\circ_\lambda)$ is an associative conformal algebra, and
  $\rightharpoonup_\lambda$, $\lhd_\lambda$,
$\circ_\lambda$ satisfy
\begin{eqnarray}
&&a\rightharpoonup_\lambda(x\circ_\mu y)=(a\rightharpoonup_\lambda x)\circ_{\lambda+\mu}y,~~~~x\circ_\lambda(a\rightharpoonup_\mu y)
=(x\lhd_\lambda a)\circ_{\lambda+\mu}y,\\
&&x\circ_\lambda(y\lhd_\mu a)=(x\circ_\lambda y)\lhd_{\lambda+\mu}a,
\end{eqnarray}
where $a\in A$ and $x$, $y\in Q$. Denote this unified product by $A\ltimes Q$, which is called the \emph{semidirect sum}
of algebras $A$ and $Q$.
\begin{corollary}
Let $A$  be an associative conformal algebra and $Q$ be a $\mathbb{C}[\partial]$-module.
Set $E= A\oplus Q$, where the direct sum is the sum of $\mathbb{C}[\partial]$-modules.
Assume that $E$ has an associative conformal algebra structure such that $A$ is a subalgebra of $E$ and
the inclusion $A\hookrightarrow E$ has a retraction that is an algebra map. Then
$E$ is isomorphic to a semidirect sum  $A\ltimes Q$.
\end{corollary}
\begin{proof}
It can be directly obtained from Theorem \ref{t2}.
\end{proof}

\subsection{Crossed product and bicrossed product}
Let $\Omega(A,Q)=(\leftharpoonup_\lambda,\rightharpoonup_\lambda, \lhd_\lambda,$ $\rhd_\lambda,
f_\lambda,\circ_\lambda )$ be an extending datum of an associative conformal algebra $A$ by a $\mathbb{C}[\partial]$-module $Q$ where
$\rightharpoonup_\lambda$ and $\lhd_\lambda$ is trivial. Then $\Omega(A,Q)=(\leftharpoonup_\lambda, \rhd_\lambda, f_\lambda, \circ_\lambda )$ is an associative conformal extending structure of $A$ by $Q$ if and only if
$(Q,\circ_\lambda)$ is an associative conformal algebra and
the following compatibility conditions are satisfied for all $a$, $b\in A$, $x$, $y$, $z\in Q$:
\begin{eqnarray}
(a_\lambda b)\leftharpoonup_{\lambda+\mu}x=a_\lambda(b\leftharpoonup_\mu x),\\
a_\lambda(x\rhd_\mu b)
=(a\leftharpoonup_\lambda x)_{\lambda+\mu}b,\\
x\rhd_\lambda(a_\mu b)=(x\rhd_\lambda a)_{\lambda+\mu}b,\\
a\leftharpoonup_\lambda(x\circ_\mu y)=(a\leftharpoonup_\lambda x)\leftharpoonup_{\lambda+\mu}y-a_\lambda f_\mu(x,y),\\
x\rhd_\lambda (a\leftharpoonup_\mu y)=(x\rhd_\lambda a)\leftharpoonup_{\lambda+\mu}y,\\
x\rhd_\lambda(y\rhd_\mu a)=f_\lambda(x,y)_{\lambda+\mu}a
+(x\circ_\lambda y)\rhd_{\lambda+\mu}a,\\
x\rhd_\lambda f_\mu(y,z)
+f_\lambda(x,y\circ_\mu z)=f_\lambda(x,y)\leftharpoonup_{\lambda+\mu}z+f_{\lambda+\mu}(x\circ_\lambda y,z).
\end{eqnarray}
The associated unified product denoted by $A\natural_{\leftharpoonup,\rhd}^f Q$ is called the \emph{crossed product}
of $A$ and $Q$. Note that $A$ is an ideal of $A\natural_{\leftharpoonup,\rhd}^fQ$.

\begin{proposition}\label{p5}
Let $A$, $Q$ be two associative conformal algebras.
Set $E= A\oplus Q$, where the direct sum is the sum of $\mathbb{C}[\partial]$-modules. Assume that $E$ has an associative conformal algebra structure  such that $A$ is an ideal of $E$. Then
$E$ is isomorphic to a crossed product $A\natural_{\leftharpoonup,\rhd}^f Q$ of associative conformal algebras $A$ and $Q$.
\end{proposition}
\begin{proof}
It can be directly obtained by Theorem \ref{t2}.
\end{proof}

Therefore, by the discussion that in Section 3, the crossed product of $A$ and $Q$ can be used to give an answer to the $\mathbb{C}[\partial]$-split extension problem about $A$ and $Q$. A similar cohomological object as $\mathcal{AH}^2(Q,A)$
can be given to describe and characterize the $\mathbb{C}[\partial]$-split extension problem. The details can be referred to that in Section 3.

Let $\Omega(A,Q)=(\leftharpoonup_\lambda,\rightharpoonup_\lambda, \lhd_\lambda,\rhd_\lambda,
f_\lambda,\circ_\lambda )$ be an extending datum of an associative conformal algebra $A$ by a $\mathbb{C}[\partial]$-module $Q$ where
$f_\lambda$ is trivial. Then $\Omega(A,Q)=(\leftharpoonup_\lambda,\rightharpoonup_\lambda, \lhd_\lambda,\rhd_\lambda,\circ_\lambda )$ is an associative conformal extending structure of $A$ by $Q$ if and only if
$(Q,\circ_\lambda)$ is an associative conformal algebra,
$(Q,\rightharpoonup_\lambda,\lhd_\lambda)$ is an $A$-bimodule, $(A,\rhd_\lambda, \leftharpoonup_\lambda)$ is a $Q$-bimodule, and
the following compatibility conditions hold for all $a$, $b\in A$, $x$, $y\in Q$:
\begin{eqnarray}
&&(a_\lambda b)\leftharpoonup_{\lambda+\mu}x=a_\lambda(b\leftharpoonup_\mu x)+a\leftharpoonup_\lambda(b\rightharpoonup_\mu x),\\
&&a_\lambda(x\rhd_\mu b)+a\leftharpoonup_\lambda(x\lhd_\mu b)
=(a\leftharpoonup_\lambda x)_{\lambda+\mu}b+(a\rightharpoonup_\lambda x)\rhd_{\lambda+\mu}b,\\
&&x\rhd_\lambda(a_\mu b)=(x\rhd_\lambda a)_{\lambda+\mu}b+(x\lhd_\lambda a)\rhd_{\lambda+\mu}b,\\
&&a\rightharpoonup_\lambda(x\circ_\mu y)=(a\rightharpoonup_\lambda x)\circ_{\lambda+\mu}y+(a\leftharpoonup_\lambda x)\rightharpoonup_{\lambda+\mu} y,\\
&&x\lhd_\lambda(a\leftharpoonup_\mu y)+x\circ_\lambda(a\rightharpoonup_\mu y)
=(x\rhd_\lambda a)\rightharpoonup_{\lambda+\mu}y+(x\lhd_\lambda a)\circ_{\lambda+\mu}y,\\
&&x\lhd_\lambda(y\rhd_\mu a)+x\circ_\lambda(y\lhd_\mu a)
=(x\circ_\lambda y)\lhd_{\lambda+\mu}a.
\end{eqnarray}
The associated unified product is denoted by $A\bowtie Q$, which is called the \emph{bicrossed product} of $A$ and $Q$.
Note that both $A$ and $Q$ are subalgebras of $A\bowtie Q$.

\begin{proposition}\label{p4}
Let $A$, $Q$ be two associative conformal algebras.
Set $E= A\oplus Q$, where the direct sum is the sum of $\mathbb{C}[\partial]$-modules. Assume that $E$ has an associative conformal algebra structure  such that $A$ and $Q$ are two subalgebras of $E$. Then
$E$ is isomorphic to a bicrossed product $A\bowtie Q$ of associative conformal algebras $A$ and $Q$.
\end{proposition}
\begin{proof}
It can be directly obtained by Theorem \ref{t2}.
\end{proof}

\subsection{Unified product of commutative conformal algebras}
In this subsection,  the unified product of commutative conformal algebras is considered. According to (\ref{er1})
and (\ref{er2}), we can obtain that a unified product $A\natural Q$ is commutative if and only if
$A$ is commutative, $f_\lambda: Q\times Q\rightarrow A[\lambda]$, $\circ_\lambda: Q\times Q\rightarrow Q[\lambda]$ are two symmetric  conformal bilinear maps, $a\leftharpoonup_\lambda x=x\rhd_{-\lambda-\partial}a$ and
$a\rightharpoonup_\lambda x=x\lhd_{-\lambda-\partial}a$ for all $a\in A$ and $x\in Q$.
Therefore, we introduce the definition of unified product of commutative conformal algebras as follows.

\begin{definition}
Let $(A,\cdot_\lambda \cdot)$ be a commutative conformal algebra  and $Q$ a $\mathbb{C}[\partial]$-module. An \emph{extending datum} of $A$ by $Q$ is
a system $\Omega(A, Q)=(\lhd_\lambda, \rhd_\lambda,  f_\lambda, \circ_\lambda )$ consisting four conformal bilinear maps given by
\begin{eqnarray*}
\lhd_\lambda: Q\times A\rightarrow Q[\lambda],~~~\rhd_\lambda: Q\times A\rightarrow A[\lambda],\\
f_\lambda: Q\times Q \rightarrow A[\lambda],~~~\circ_\lambda: Q\times Q \rightarrow Q[\lambda].
\end{eqnarray*}
Let $\Omega(A, Q)=(\lhd_\lambda, \rhd_\lambda,  f_\lambda, \circ_\lambda )$ be an extending datum. Denote by $A\natural_{\Omega(A, Q)}Q=A\natural Q$ the $\mathbb{C}[\partial]$-module
$A\times Q$ with the natural $\mathbb{C}[\partial]$-module action: $\partial (a, x)=(\partial a, \partial x)$ and the bilinear map $
\cdot_\lambda \cdot: (A\times Q)\times (A\times Q)\rightarrow (A\times Q)[\lambda]$ defined by
\begin{gather}
(a,x)_\lambda (b,y)\nonumber\\
\label{ce3}=(a_\lambda b+y\rhd_{-\lambda-\partial} a+x\rhd_{\lambda} b+f_\lambda(x,y),y\lhd_{-\lambda-\partial} a+x\lhd_\lambda b+x\circ_\lambda y),
\end{gather}
for all $a$, $b\in A$, $x$, $y\in Q$. Obviously, since $\lhd_\lambda$, $\rhd_\lambda$, $f_\lambda$ and $\circ_\lambda$ are conformal bilinear maps, the $\lambda$-product defined by (\ref{ce3}) satisfies conformal sesquilinearity. Then
$A\natural Q$ is called the \emph{unified product} of $A$ and $Q$ associated with $\Omega(A, Q)$ if it is a commutative conformal algebra with the $\lambda$-product given by (\ref{ce3}). In this case, the extending datum $\Omega(A, Q)$ is called a \emph{commutative conformal extending structure} of $A$ by $Q$.
\end{definition}

Similarly, we can give a theorem in the commutative case as Theorem \ref{t1}.

\begin{theorem}\label{ct1}
Let $A$ be a commutative conformal algebra, $Q$ be a $\mathbb{C}[\partial]$-module and $\Omega(A, Q)$ be an extending datum of $A$ by $Q$.
Then $A\natural Q$ is a commutative conformal algebra if and only if the following conditions are satisfied for all
$a$, $b\in A$ and $x$, $y$, $z\in Q$:
\begin{eqnarray*}
(CCE1)&&~~f_\lambda(x,y)=f_{-\lambda-\partial}(y,x),~~~~~~x\circ_\lambda y=y\circ_{-\lambda-\partial}x,\\
(CCE2)&&~~a_\lambda(x\rhd_{-\mu-\partial}b)+(x\lhd_{-\mu-\partial}b)\rhd_{-\lambda-\partial}a=x\rhd_{-\lambda-\mu-\partial}(a_\lambda b),\\
(CCE3)&&~~(x\lhd_{-\mu-\partial}b)\lhd_{-\lambda-\partial}a=x\lhd_{-\lambda-\mu-\partial}(a_\lambda b),\\
(CCE4)&&~~a_\lambda f_\mu(x,y)+(x\circ_\mu y)\rhd_{-\lambda-\partial}a=y\rhd_{-\lambda-\mu-\partial}(x\rhd_{-\lambda-\partial}a)\\
&&+f_{\lambda+\mu}(x\lhd_{-\lambda-\partial}a,y),\\
(CCE5)&&~~(x\circ_\mu y)\lhd_{-\lambda-\partial}a=y\lhd_{-\lambda-\mu-\partial}(x\rhd_{-\lambda-\partial}a)
+(x\lhd_{-\lambda-\partial}a)\circ_{\lambda+\mu}y,\\
(CCE6)&&~~x\rhd_\lambda f_\mu(y,z)+f_\lambda(x,y\circ_\mu z)=z\rhd_{-\lambda-\mu-\partial} f_\lambda(x,y)
+f_{\lambda+\mu}(x\circ_\lambda y,z),\\
(CCE7)&&~~x\lhd_\lambda f_\mu(y,z)+x\circ_\lambda(y\circ_\mu z)
=z\lhd_{-\lambda-\mu-\partial}(x\circ_\lambda y)+(x\circ_\lambda y)\circ_{\lambda+\mu}z.
\end{eqnarray*}

\end{theorem}
\begin{proof}
The proof is similar to that in Theorem \ref{t1}.
\end{proof}

\section{Unified products when $Q=\mathbb{C}[\partial]x$}

Set $A$ be an associative conformal algebra, $Q$ be a $\mathbb{C}[\partial]$-module and $E$ be that in the $\mathbb{C}[\partial]$-split extending structures problem. When $Q$ is a torsion $\mathbb{C}[\partial]$-module, due to the fact that the torsion element must be a central element, $Q$ is a two-sided ideal contained in the center of $E$. Therefore, in this case, $E\cong A\oplus Q$ as associative conformal algebras, and $Q$ is a trivial two-sided ideal of $E$.

Similarly, if $A$ is a torsion $\mathbb{C}[\partial]$-module, the associative conformal algebra structure on $E$ is just the central extension of $Q$ by $A$.

Next, we will use the general theory developed in Section 3 to study the case when $A$ is a free $\mathbb{C}[\partial]$-module and $Q$ is a free $\mathbb{C}[\partial]$-module of rank 1. Set $Q=\mathbb{C}[\partial]x$.
\begin{definition}
Let $A=\mathbb{C}[\partial]V$ be an associative conformal algebra which is free as a $\mathbb{C}[\partial]$-module. A \emph{flag datum} of $A$ is
a 6-tuple $(h_\lambda(\cdot,\partial),$ $ D_\lambda,$ $ g_\lambda(\cdot,\partial),$
$T_\lambda, $ $Q_0(\lambda,\partial), $ $P(\lambda,\partial))$ where $P(\lambda,\partial)\in \mathbb{C}[\lambda,\partial]$,
$Q_0(\lambda,\partial)\in A[\lambda]$, $h_\lambda(\cdot,\partial):A\rightarrow \mathbb{C}[\lambda,\partial]$ and $D_\lambda: A\rightarrow A[\lambda]$ are two left conformal linear maps, and $g_\lambda(\cdot,\partial): A\rightarrow \mathbb{C}[\lambda,\partial]$ and $T_\lambda: A\rightarrow A[\lambda]$ are two conformal linear maps
satisfying the following conditions $(a, b\in V)$:
\begin{gather}
\label{f1}D_{\lambda+\mu}(a_\lambda b)=a_\lambda(D_\mu(b))+h_\mu(b,\lambda+\partial)D_\lambda(a),\\
\label{f2}h_\mu(b,\lambda+\partial)h_\lambda(a,\partial)=h_{\lambda+\mu}(a_\lambda b,\partial),\\
\label{f3}a_\lambda(T_\mu(b))+g_\mu(b,\lambda+\partial)D_\lambda(a)=(D_\lambda(a))_{\lambda+\mu}b
+h_\lambda(a,-\lambda-\mu)T_{\lambda+\mu}(b),\\
\label{f4}g_\mu(b,\lambda+\partial)h_\lambda(a,\partial)=h_\lambda(a,-\lambda-\mu)g_{\lambda+\mu}(b,\partial),\\
\label{f5}T_\lambda(a_\mu b)=(T_\lambda(a))_{\lambda+\mu}b+g_\lambda(a,-\lambda-\mu)T_{\lambda+\mu}(b),\\
\label{f6}g_\lambda(a_\mu b,\partial)=g_\lambda(a,-\lambda-\mu)g_{\lambda+\mu}(b,\partial),\\
\label{f7}P(\mu,\lambda+\partial)D_\lambda(a)=D_{\lambda+\mu}(D_\lambda(a))
+h_\lambda(a,-\lambda-\mu)Q_0(\lambda+\mu,\partial)-a_\lambda(Q_0(\mu,\partial)),\\
\label{f8} P(\mu,\lambda+\partial)h_\lambda(a,\partial)=h_\lambda(a,-\lambda-\mu)P(\lambda+\mu,\partial)
+h_{\lambda+\mu}(D_\lambda(a),\partial),\\
\label{f9}T_\lambda(D_\mu(a))+h_\mu(a,\lambda+\partial)Q_0(\lambda,\partial)
=D_{\lambda+\mu}(T_\lambda(a))+g_\lambda(a,-\lambda-\mu)Q_0(\lambda+\mu,\partial),\\
\label{f10}g_\lambda(D_\mu(a),\partial)
+h_\mu(a,\lambda+\partial)P(\lambda,\partial)=h_{\lambda+\mu}(T_\lambda(a),\partial)
+g_\lambda(a,-\lambda-\mu)P(\lambda+\mu,\partial),\\
\label{f11}T_\lambda(T_\mu(a))+g_\mu(a,\lambda+\partial)Q_0(\lambda,\partial)
=Q_0(\lambda,\partial)_{\lambda+\mu}a+P(\lambda,-\lambda-\mu)T_{\lambda+\mu}(a),\\
\label{f12}g_\lambda(T_\mu(a),\partial)+g_\mu(a,\lambda+\partial)P(\lambda,\partial)=P(\lambda,-\lambda-\mu)g_{\lambda+\mu}(a,\partial),\\
\label{f13}T_\lambda(Q_0(\mu,\partial))+P(\mu,\lambda+\partial)Q_0(\lambda,\partial)
=D_{\lambda+\mu}(Q_0(\lambda,\partial))+P(\lambda,-\lambda-\mu)Q_0(\lambda+\mu,\partial),\\
\label{f14}g_\lambda(Q_0(\mu,\partial),\partial)+P(\mu,\lambda+\partial)P(\lambda,\partial)
=h_{\lambda+\mu}(Q_0(\lambda,\partial),\partial)+P(\lambda,-\lambda-\mu)P(\lambda+\mu,\partial).
\end{gather}

\end{definition}

Denote the set of all flag datums of $A$ by $\mathcal{FAC}(A)$.

\begin{proposition}\label{pr1}
Let $A=\mathbb{C}[\partial]V$ be an associative conformal algebra which is a free $\mathbb{C}[\partial]$-module and $Q=\mathbb{C}[\partial]x$ be a free $\mathbb{C}[\partial]$-module of rank 1.
Then there is a bijection between the set  $\mathcal{TC}(A,Q)$ of all associative conformal extending structures of $A$ by $Q$ and $\mathcal{FAC}(A)$.
\end{proposition}
\begin{proof}
Let $\Omega(A,Q)=(\leftharpoonup_\lambda,\rightharpoonup_\lambda,\lhd_\lambda,\rhd_\lambda,f_\lambda,
\circ_\lambda)$ be an associative conformal extending structure. According to the condition that $Q=\mathbb{C}[\partial]x$ is a free
$\mathbb{C}[\partial]$-module of rank 1, we can set
\begin{eqnarray*}
&& a\rightharpoonup_\lambda x=h_\lambda(a,\partial)x,~~a\leftharpoonup_\lambda x=D_\lambda(a),\\
&&x\lhd_\lambda a=g_\lambda(a,\partial)x,~~x\rhd_\lambda a=T_\lambda(a),\\
&&x\circ_\lambda x=P(\lambda,\partial)x,~~f_\lambda(x,x)=Q_0(\lambda,\partial),
\end{eqnarray*}
where $a\in V$, $P(\lambda,\partial)\in \mathbb{C}[\lambda,\partial]$, $Q_0(\lambda,\partial)\in A[\lambda]$ and $h_\lambda(\cdot,\partial)$, $g_\lambda(\cdot,\partial): A\rightarrow \mathbb{C}[\lambda,\partial]$ and $D_\lambda$, $T_\lambda: A\rightarrow A[\lambda]$ are four linear maps.

Since $\leftharpoonup_\lambda$, $\rightharpoonup_\lambda$, $\lhd_\lambda$ and $\rhd_\lambda$ are four conformal bilinear maps, we can get $h_\lambda(\cdot,\partial)$ and $D_\lambda$ are left conformal linear maps, and $g_\lambda(\cdot,\partial)$, and $T_\lambda$
are two conformal linear maps. Moreover, it is easy to check that
$(ACE1)$-$(ACE14)$ hold if and only if $(\ref{f1})$-$(\ref{f14})$ are satisfied.
\end{proof}

By Proposition \ref{pr1}, the associative conformal algebra corresponding
to the flag datum $(h_\lambda(\cdot,\partial), D_\lambda, g_\lambda(\cdot,\partial),
T_\lambda, Q_0(\lambda,\partial), P(\lambda,\partial))$ of $A$ is the $\mathbb{C}[\partial]$-module
$A\oplus \mathbb{C}[\partial]x$ with the following $\lambda$-product
\begin{eqnarray}
&&(a,0)_\lambda (b,0)=(a_\lambda b,0),~~(0,x)_\lambda(0,x)=(Q_0(\lambda,\partial),P(\lambda,\partial)x),\\
&&(a,0)_\lambda(0,x)=(D_{\lambda}(a),h_{\lambda}(a,\partial)x),~~(0,x)_\lambda(a,0)
=(T_\lambda(a),g_\lambda(a,\partial)x),
\end{eqnarray}
for any $a$, $b\in V$. We denote this associative conformal algebra by $AC(A,\mathbb{C}[\partial]x|$ $(h_\lambda(\cdot,\partial), $ $D_\lambda,$ $g_\lambda(\cdot,\partial),$
$T_\lambda,$ $ Q_0(\lambda,\partial),$ $ P(\lambda,\partial))$.

\begin{theorem}\label{t5}
Let $A=\mathbb{C}[\partial]V$ be an associative conformal algebra and $Q=\mathbb{C}[\partial]x$ be a free $\mathbb{C}[\partial]$-module of rank 1. Set $E=A\oplus Q$ as a $\mathbb{C}[\partial]$-module. Then we have\\
(1) $CExtd(E,A)\cong\mathcal{AH}_A^2(Q,A)\cong \mathcal{FAC}(A)/\equiv$,
where $\equiv$ is the equivalence relation on the set
$\mathcal{FAC}(A)$ as follows:
$(h_\lambda(\cdot,\partial), D_\lambda, g_\lambda(\cdot,\partial),
T_\lambda, Q_0(\lambda,\partial), P(\lambda,\partial))\equiv (h_\lambda^{'}(\cdot,\partial),$ $D_\lambda^{'},$ $g_\lambda^{'}(\cdot,\partial), $ $T_\lambda^{'},$ $Q_0^{'}(\lambda,\partial),$ $P^{'}(\lambda,\partial)),$
if and only if
$h_\lambda(\cdot,\partial)=h_\lambda^{'}(\cdot,\partial)$, $g_\lambda(\cdot,\partial)=g_\lambda^{'}(\cdot,\partial)$ and there exist $T_0\in A$ and $\beta\in \mathbb{C}\backslash\{0\}$ such that for all $a\in A$:
\begin{gather}
\label{tt1}D_\lambda(a)=a_\lambda T_0+\beta D_\lambda^{'}(a)-h_\lambda(a,\partial)T_0,\\
T_\lambda(a)={T_0}_\lambda a+\beta T_\lambda^{'}(a)-g_\lambda(a,\partial)T_0,\\
P(\lambda,\partial)=h^{'}_\lambda(T_0,\partial)+g_\lambda^{'}(T_0,\partial)+\beta P^{'}(\lambda,\partial),\\
\label{tt2}Q_0(\lambda,\partial)={T_0}_\lambda T_0+\beta D_\lambda^{'}(T_0)
+\beta T_\lambda^{'}(T_0)+\beta^2Q_0^{'}(\lambda,\partial)-P(\lambda,\partial)T_0.
\end{gather}
The bijection between $\mathcal{FAC}(A)/\equiv$ and
$CExtd(E,A)$ is given by
$$\overline{(h_\lambda(\cdot,\partial), D_\lambda, g_\lambda(\cdot,\partial),
T_\lambda, Q_0(\lambda,\partial), P(\lambda,\partial))}\rightarrow$$
$$AC(A,\mathbb{C}[\partial]x|h_\lambda(\cdot,\partial), D_\lambda, g_\lambda(\cdot,\partial),
T_\lambda, Q_0(\lambda,\partial), P(\lambda,\partial)).$$
(2) $CExtd^{'}(E,A)\cong\mathcal{AH}^2(Q,A)\cong \mathcal{FAC}(A)/\approx$,
where $\approx$ is the equivalence relation on the set
$\mathcal{FAC}(A)$ as follows:
$(h_\lambda(\cdot,\partial), D_\lambda, g_\lambda(\cdot,\partial),
T_\lambda, Q_0(\lambda,\partial), P(\lambda,\partial))\approx (h_\lambda^{'}(\cdot,\partial),$ $D_\lambda^{'},$ $g_\lambda^{'}(\cdot,\partial), $$T_\lambda^{'},$ $Q_0^{'}(\lambda,\partial),$ $P^{'}(\lambda,\partial)),$ if and only if
$h_\lambda(\cdot,\partial)=h_\lambda^{'}(\cdot,\partial)$, $g_\lambda(\cdot,\partial)=g_\lambda^{'}(\cdot,\partial)$ and there exists $T_0\in A$  such that (\ref{tt1})-(\ref{tt2}) hold for $\beta=1$.
The bijection between $\mathcal{FAC}(A)/\approx$ and
$CExtd^{'}(E,A)$ is given by
$$\overline{\overline{(h_\lambda(\cdot,\partial), D_\lambda, g_\lambda(\cdot,\partial),
T_\lambda, Q_0(\lambda,\partial), P(\lambda,\partial))}}\rightarrow$$
$$AC(A,\mathbb{C}[\partial]x|h_\lambda(\cdot,\partial), D_\lambda, g_\lambda(\cdot,\partial),
T_\lambda, Q_0(\lambda,\partial), P(\lambda,\partial)).$$
\end{theorem}
\begin{proof}
Since in Lemma \ref{l1}, $u: Q\rightarrow A$ is a $\mathbb{C}[\partial]$-module
homomorphism and $v\in \text{Aut}_{\mathbb{C}[\partial]}(Q)$,
we set $u(x)=T_0$ and $v(x)=\beta x$ where $T_0\in A$ and $\beta\in \mathbb{C}\backslash\{0\}$. Then this theorem can be directly obtained from Lemma \ref{l1}, Theorem \ref{t4} and
Proposition \ref{pr1}.
\end{proof}

Finally, we use this theory to characterize the extending structures of associative conformal algebras which is free and of rank 1 as a $\mathbb{C}[\partial]$-module by $Q=\mathbb{C}[\partial]x$.

\begin{proposition}\label{hh2}
Let $A=\mathbb{C}[\partial]e$ be the associative conformal algebra with $e_\lambda e=0$ and $Q=\mathbb{C}[\partial]x$. Then $\mathcal{AH}_{A}^2(Q,A)$ can be described by the following five kinds of flag datums:\\
(1) $(0,0,0,0,Q(\lambda,\partial)e,0)$,
where $Q(\lambda,\partial)\in \mathbb{C}[\lambda,\partial]$. Moreover, $(0,0,0,0,Q(\lambda,\partial)e,0)$ is equivalent to $(0,0,0,0,Q^{'}(\lambda,\partial)e,0)$ if and only if there exists $\beta\in \mathbb{C}\setminus\{0\}$ such that
$Q(\lambda,\partial)=\beta Q^{'}(\lambda,\partial)$;\\
(2) $(0,0,0,0,0,1)$;\\
(3) $(0,0,0,T_\lambda^1,0,1)$ where $T_\lambda^1(e)= e$;\\
(4)  $(0,D_\lambda^1,0,0,0,1)$ where $D_\lambda^1(e)= e$;\\
(5) $(0,D_\lambda^1,0,T_\lambda^1,0,1)$ where $D_\lambda^1(e)=T_\lambda^1(e)= e$.
\end{proposition}
\begin{proof}
By Theorem \ref{t5}, for characterizing the extending structures of  $A$ by $Q$ up to equivalence, we only need to describe the set $\mathcal{FAC}(A)$ up to equivalence.

Since $A$ is a trivial associative conformal algebra of rank 1, $E=A\natural Q$ is an associative conformal algebra of rank 2 as a free $\mathbb{C}[\partial]$-module with a nilpotent element. Note that there is an analogue of the Wedderburn Theorem in \cite{Z2}  that  an arbitrary finite associative conformal algebra $C$ can be  presented as $C=S\oplus R$, where
$R$ is the maximal nilpotent ideal of $C$ and $S$ is a semisimple subalgebra isomorphic to $C/R$. Therefore, according to the
Wedderburn Theorem, $E$ must have a nonzero maximal nilpotent ideal. By Remark \ref{rr1}, $(Q,\rightharpoonup_\lambda,
\lhd_\lambda)$ is an $A$-bimodule. Since $A$ is trivial, $\rightharpoonup_\lambda$ and $\lhd_\lambda$ are trivial.
Therefore, $h_\lambda(e,\partial)=g_\lambda(e,\partial)=0$. Next, we will give a discussion about whether $E$ has a nonzero semisimple part.

If there is a nonzero semisimple part in $E$, then it is $\text{Cur}_1$ and the maximal nilpotent ideal is $A$.
By the Wedderburn Theorem, $Q$ may be shifted by $A$ to get $\text{Cur}_1$. So, by Theorem \ref{t5}, $P(\lambda,\partial)=1$ and $Q_0(\lambda,\partial)=0$ up to equivalence and $(A,\leftharpoonup_\lambda,\rhd_\lambda)$ is a bimodule over
$Q=\mathbb{C}[\partial]x$, where $x\circ_\lambda x=x$. According to that $A$ is a bimodule over
$Q=\mathbb{C}[\partial]x$ through the actions of $T_\lambda$ and $D_\lambda$, it is easy to see that there are four cases, i.e. $T_\lambda(e)=D_\lambda(e)=0$, $T_\lambda(e)=e$, $D_\lambda(e)=0$,
$T_\lambda(e)=0$, $D_\lambda(e)=e$, and $T_\lambda(e)=D_\lambda(e)=e$. It is easy to see that the first case is just Case (2), the second case is just  Case (3), the third case is just Case (4) and the fourth case is just Case (5).

If there is no nonzero semisimple part in $E$, then $E$ is nilpotent. Therefore, $P(\lambda,\partial)=0$.
By (\ref{f7}) and (\ref{f11}), we can easily get that $D_\lambda(e)=T_\lambda(e)=0$.
 Then,  for any $Q_0(\lambda,\partial)=Q(\lambda,\partial)e\in A[\lambda]$,
(\ref{f1})-(\ref{f14}) hold.  Therefore, in this case, all flag datums are of the form $(0,0,0,0,Q(\lambda,\partial)e,0)$,
where $Q(\lambda,\partial)\in \mathbb{C}[\lambda,\partial]$. By Theorem \ref{t5}, $(0,0,0,0,Q(\lambda,\partial)e,0)$ is equivalent to $(0,0,0,0,Q^{'}(\lambda,\partial)e,0)$ if and only if there exists $\beta\in \mathbb{C}\setminus\{0\}$ such that
$Q(\lambda,\partial)=\beta^2 Q^{'}(\lambda,\partial)$. Note that $\mathbb{C}$ is a perfect field. Therefore, this is Case (1).

Obviously, by Theorem \ref{t5}, the five kinds of flag datums are not equivalent to each other. 
\end{proof}

\begin{proposition}\label{hh3}
Let $A=\mathbb{C}[\partial]e$ be the associative conformal algebra with $e_\lambda e=e$ and $Q=\mathbb{C}[\partial]x$. Then $\mathcal{AH}_{A}^2(Q,A)$ can be described by the following seven kinds of flag datums:\\
(1) $(0,0,0,0,0,0)$;\\
(2) $(0,0,0,0,0,1)$;\\
(3) $(h_\lambda^1(\cdot,\partial),0,0,0,0,0)$ where $h_\lambda^1(e,\partial)= 1$;\\
(4)  $(0,0,g_\lambda^1(\cdot,\partial),0,0,0)$ where $g_\lambda^1(e,\partial)= 1$;\\
(5) $(h_\lambda^1(\cdot,\partial),0,g_\lambda^1(\cdot,\partial),0,0,0)$ where $h_\lambda^1(e,\partial)=g_\lambda^1(e,\partial)= 1$; \\ (6)$(h_\lambda^1(\cdot,\partial),0,g_\lambda^1(\cdot,\partial),0,e,0)$  where $h_\lambda^1(e,\partial)=g_\lambda^1(e,\partial)= 1$;\\
(7)$(h_\lambda^1(\cdot,\partial),0,g_\lambda^1(\cdot,\partial),0,ce,1)$ for any $c\in \mathbb{C}$,  where $h_\lambda^1(e,\partial)=g_\lambda^1(e,\partial)= 1$.
\end{proposition}
\begin{proof}
Similar to that in Proposition \ref{hh2}, we only need to describe the set $\mathcal{FAC}(A)$ up to equivalence.

Since $(Q,\rightharpoonup_\lambda,\lhd_\lambda)$ is an $A$-bimodule, it is easy to get that
$e\rightharpoonup_\lambda x=x$ or $e\rightharpoonup_\lambda x=0$, and $x\lhd_\lambda e=x$ or $x\lhd_\lambda e=0$.
Therefore, $h_\lambda(e,\partial)=g_\lambda(e,\partial)=0$, or $h_\lambda(e,\partial)=1$, $g_\lambda(e,\partial)=0$,
or $h_\lambda(e,\partial)=0$, $g_\lambda(e,\partial)=1$, or $h_\lambda(e,\partial)=g_\lambda(e,\partial)=1$.

Next, by the Wedderburn Theorem, we will give a discussion about whether $E$ has a nonzero maximal nilpotent ideal.

If $E$ has a nonzero maximal nilpotent ideal, then by Theorem \ref{t5}, we can make $Q$ be the nilpotent ideal. Therefore, in this case,
$P(\lambda,\partial)=0$,  $Q_0(\lambda,\partial)=0$ up to equivalence and $(A,\leftharpoonup_\lambda,\rhd_\lambda)$ is a bimodule over
$Q=\mathbb{C}[\partial]x$, where $x\circ_\lambda x=0$. Since $Q$ is a trivial associative conformal algebra,
$\leftharpoonup_\lambda$ and $\rhd_\lambda$ are trivial. Therefore, $T_\lambda(e)=D_\lambda(e)=0$.
Thus, according to the discussion about $h_\lambda(\cdot,\partial)$ and $g_\lambda(\cdot,\partial)$,
we can get Case (1), Case (3), Case (4) and Case (5).

If $E$ has no nonzero maximal nilpotent ideal, then $E$ is semisimple by the Wedderburn Theorem. Therefore,
as an associative conformal algebra, $E$ is isomorphic to $\text{Cur}_1\oplus \text{Cur}_1$.
Then we discuss it in two cases, i.e. $P(\lambda,\partial)\neq 0$ and $P(\lambda,\partial)=0$.\\
{\bf Subcase 1}: $P(\lambda,\partial)\neq 0$. According to $E\cong \text{Cur}_1\oplus \text{Cur}_1$ and
by Theorem \ref{t5}, we can get $P(\lambda,\partial)=1$ up to equivalence. Set $T_\lambda(e)=T(\lambda,\partial)e$,
and $D_\lambda(e)=D(\lambda,\partial)e$. Setting $a=b=e$ in (\ref{f5}) and (\ref{f12}),
we can get
\begin{eqnarray}
T(\lambda,\partial)=T(\lambda,-\lambda-\mu)+g_\lambda(e,-\lambda-\mu)T(\lambda+\mu,\partial),\\
T(\mu,\lambda+\partial)g_\lambda(e,\partial)+g_\mu(e,\lambda+\partial)=g_{\lambda+\mu}(e,\partial).
\end{eqnarray}
Therefore, we can directly obtain that when $g_\lambda(e,\partial)=0$, $T(\lambda,\partial)=T(\lambda)$ for
some $T(\lambda)\in \mathbb{C}[\lambda]$; when $g_\lambda(e,\partial)=1$, $T(\lambda,\partial)=0$.
Similarly, by (\ref{f1}) and (\ref{f8}), we can get that when $h_\lambda(e,\partial)=0$, $D(\lambda,\partial)=D(\lambda+\partial)$ for
some $D(\lambda)\in \mathbb{C}[\lambda]$; when $h_\lambda(e,\partial)=1$, $D(\lambda,\partial)=0$.
According to (\ref{f3}) and (\ref{f10}), one can obtain that
when $h_\lambda(e,\partial)=1$ and $g_\lambda(e,\partial)=0$,
$T_\lambda(e)=e$ and $D_\lambda(e)=0$;
when $h_\lambda(e,\partial)=0$ and $g_\lambda(e,\partial)=1$,
$T_\lambda(e)=0$ and $D_\lambda(e)=e$; when $h_\lambda(e,\partial)=g_\lambda(e,\partial)=0$,
$T_\lambda(e)=p(\lambda)e$ and $D_\lambda(e)=p(-\lambda-\partial)e$ for some $p(\lambda)\in \mathbb{C}[\lambda]$; when $h_\lambda(e,\partial)=g_\lambda(e,\partial)=1$,
$T_\lambda(e)=D_\lambda(e)=0$. By (\ref{f9}), the first two cases can not hold. When $h_\lambda(e,\partial)=g_\lambda(e,\partial)=0$, by (\ref{f7}), one can get
$Q_0(\lambda,\partial)=(p(\lambda)p(-\lambda-\partial)-p(-\partial))e$. Therefore, in this case,
all flag datums are of the form $(0,D_\lambda^p,0,T_\lambda^p,Q^{p}(\lambda,\partial) e,1)$,
where $D_\lambda^p(e)=p(-\lambda-\partial) e$, $T_\lambda^p(e)=p(\lambda) e$ and $Q^{p}(\lambda,\partial)=p(\lambda)p(-\lambda-\partial)- p(-\partial)$ for some $p(\lambda)\in\mathbb{C}[\lambda]$. By Theorem \ref{t5}, it is easy to get that $(0,D_\lambda^p,0,T_\lambda^p,Q^p(\lambda,\partial) e,1)$ is equivalent to $(0,0,0,0,0,1)$ with letting $T_0=p(-\partial)e$ and $\beta=1$  in  Theorem \ref{t5}. This is just Case (2). When $h_\lambda(e,\partial)=g_\lambda(e,\partial)=1$, according to $T_\lambda=D_\lambda=0$, we can directly obtain
from (\ref{f9}) and (\ref{f13}) that $Q_0(\lambda,\partial)=ce$ for some $c\in \mathbb{C}$. This is Case (7).\\
{\bf Subcase 2}: $P(\lambda,\partial)=0$. Then $Q_0(\lambda,\partial)\neq 0$. By Theorem \ref{t5} and according to that $E$ is isomorphic to $\text{Cur}_1\oplus \text{Cur}_1$, it is easy to see that we can make $Q_0(\lambda,\partial)=e$ up to equivalence. By (\ref{f14}), we can get $h_\lambda(e,\partial)=g_\lambda(e,\partial)$. According to that
$E$ is isomorphic to $\text{Cur}_1\oplus \text{Cur}_1$, one can get $h_\lambda(e,\partial)=g_\lambda(e,\partial)=1$. By
(\ref{f8}) and (\ref{f12}), it is easy to obtain that $D_\lambda=T_\lambda=0$. Therefore, this is Case (6).

Obviously, by Theorem \ref{t5}, the seven kinds of flag datums are not equivalent to each other. Therefore, we finish the proof.
\end{proof}

{\bf Acknowledgments}
{We wish to thank the referee for careful reading and useful comments. In particular, we would like to thank the referee
for suggesting the improvements of the proofs of Propositions \ref{hh2} and \ref{hh3} by using the structure theory of finite associative conformal algebras.}

\end{document}